\documentclass[12pt]{article}
\usepackage{amsmath,amsthm,amsfonts,amssymb}
\usepackage{fullpage}
\newtheorem{theorem}{Theorem}[section]
\newtheorem{lemma}{Lemma}[section]
\newtheorem{cor}{Corollary}[section]

\newtheorem{ex}{Example}[section]
\title{A note on linear preservers on semipositive and minimal semipositive matrices}

\author{Projesh Nath Choudhury \thanks{ Department of Mathematics, Indian Institute of Technology Madras, Chennai, India. Email: cn.projesh@gmail.com}  \and M. Rajesh Kannan\thanks{Department of Mathematics, Indian Institute of Technology Kharagpur, Kharagpur, India. Email: rajeshkannan@maths.iitkgp.ernet.in, rajeshkannan1.m@gmail.com } \and  K.C. Sivakumar\thanks{Department of Mathematics, Indian Institute of Technology Madras, Chennai, India. Email: kcskumar@iitm.ac.in}
}

\date{\today}
\begin{document}
\maketitle
\baselineskip=0.25in
\begin{abstract}
Semipositive matrices (matrices that map at least one nonnegative vector to a positive vector) and minimally semipositive
  matrices (semipositive matrices whose no column-deleted submatrix is semipositive) are well studied in matrix theory.
In this short note, we study the structure of linear maps which preserve the set of all semipositive and minimal semipositive matrices. 
\end{abstract}


\section{Introductory remarks}
Let $\mathbb{R}^{m\times n}$ denote the set of all $m\times n$ matrices over the real numbers. For an $m \times n$ matrix $A=(a_{ij})$; $A\geq 0$ means $a_{ij}\geq 0$. Similarly, $A>0$ means all the entries of $A$ are positive. A matrix $A\in \mathbb{R}^{m\times n}$ is said to be  \emph{semipositive},  if there exists a vector $x \geq 0$ such that $Ax > 0$. A matrix $A \in \mathbb{R}^{m\times n}$ is said to be  \emph{minimally semipositive} if it is semipositive and no column-deleted submatrix of $A$ is semipositive. By a simple perturbation argument, it follows that $A$ is semipositive if and only if there exists a vector $x\in \mathbb{R}^n$ with $x>0$ such that $Ax>0$ \cite[Lemma 2.1]{john}. Such a vector $x$ is called a \emph{semipositivity vector} of $A$. It is known that a square matrix $A$ is minimally semipositive if and only if $A^{-1}$ exists and $A^{-1}\geq 0$ \cite[Theorem 3.4]{john}. More generally, an $m\times n~(m\geq n)$ matrix $A$ is minimally semipositive if and only if $A$ is semipositive and has a nonnegative left inverse (\cite[Theorem 3.6]{john}, \cite[Theorem 2.3]{werner}). A matrix $A \in \mathbb{R}^{n\times n}$ is called  \emph{row positive} if $A\geq 0$ and each row of $A$ contains a nonzero entry. A matrix $A \in \mathbb{R}^{n\times n}$ is called  \emph{inverse nonnegative} if $A^{-1}$ exists and $A^{-1}\geq 0$. A square matrix $A$ is called \emph{monomial}, if $A\geq 0$ and every column and row of $A$ contains exactly one nonzero entry. Let us also recall a result that characterizes monomial matrices. An $n \times n$ matrix $A$ is monomial  if and only if $A$ is row positive and inverse nonnegative \cite[Lemma~2.6]{john-main}.

A linear operator $L:\mathbb{R}^{m\times n}\longrightarrow \mathbb{R}^{m\times n}$ is said to be an 
\emph{into linear preserver} of some set $S$ of matrices, if $L(S)\subseteq S$. A linear operator $L:\mathbb{R}^{m\times n}\longrightarrow \mathbb{R}^{m\times n}$ is said to be an \emph{onto linear preserver} of some set $S$ of matrices, if $L(S)=S$. There is a vast literature on linear preserver problems. For more details we refer to \cite{gut-li-semrl, li-pier}.

In \cite{john-main}, the authors considered the linear preserver problems for the set of all semipositive and minimal semipositive matrices. Consider the map $L: \mathbb{R}^{m \times n} \rightarrow \mathbb{R}^{m \times n}$ defined by $L(A)=XAY$ for some $X\in \mathbb{R}^{m \times m}$ and $Y\in \mathbb{R}^{n \times n}$.
\begin{enumerate}
\item  Then $L$ is an into linear preserver of semipositive matrices if and only if either $X$ is row positive and 
$Y$ is inverse nonnegative or $-X$ is row positive and $-Y$ is inverse nonnegative \cite[Theorem 2.4]{john-main} and
$L$ is an onto linear preserver of semipositive matrices if and only if either $X$ and $Y$ are monomial or  $-X$ and $-Y$ are monomial \cite[Corollary 2.7]{john-main}.
\item  If $m>n$, then $L$ is an into linear preserver of  minimal semipositive matrices if and only if either
$X$ is monomial and $Y$ is inverse nonnegative or $-X$ is monomial and $-Y$ is inverse nonnegative  and $L$ is an onto linear preserver of  minimal semipositive matrices if and only if either $X$ and $Y$ are monomial or  $-X$ and $-Y$ are monomial \cite[Theorem 2.11]{john-main}.
\end{enumerate}

The purpose of this short note is to revisit linear preservers of semipositive and minimal semipositive matrices. This note is organized as follows: First we show that if a vector $v\in\mathbb{R}^n$ contains both positive and negative entries and $0\neq w\in\mathbb{R}^n$, then there exists a nonnegative invertible matrix $B$ such that $Bv=w$ (Theorem \ref{np}). A similar type result for $v\geq 0$ and $w>0$ is proved in Theorem \ref{pos}. In view of the statement (2) a natural  question arises: can we get a necessary and sufficient condition for $L$ to be an into linear preserver of  minimal semipositive matrices, when $m=n$. We answer this question in Theorem \ref{mainresult1}. In Theorem \ref{mainresult2}, we give a necessary and sufficient condition for $L$ to be an onto linear preserver of  minimal semipositive matrices, when $m=n$. In the process we point to an error in \cite{john-main} and provide a correct proof.
%A matrix $A \in \mathbb{R}^{m \times n}$ is said to be \emph{nonnegative}, if all the entries of  $A$ are nonnegative. A matrix $A \in \mathbb{R}^{n \times n}$ is said to be \emph{inverse nonnegative}, if $A$ is invertible and $A^{-1}$ is nonnegative.
%\begin{theorem}\label{mono-char}
%If a matrix $A$ and its inverse are entry-wise nonnegative, then $A$ must be a monomial matrix.
%\end{theorem}

\section{Main Results}

The first result characterize onto linear preserver of some set $S$ of matrices in terms of into linear preserver. In \cite{john-main}, it is mentioned that the following result is presented in \cite{die}. For the sake of completeness, we include a proof here.
\begin{lemma}\label{basis-onto-into}
If $S$ contains a basis for $\mathbb{R}^{m\times n}$, then $L$ is an onto linear preserver of $S$ if and only if $L$ and $L^{-1}$ are into  linear preservers on $S$.
\end{lemma}
\begin{proof}
If $L$ is an onto linear preserver on $S$, then $L(S)$ contains a basis for $\mathbb{R}^{m\times n}$. Thus $L$ is invertible, $L(S) \subseteq S$ and $L^{-1}(S) \subseteq S$. Conversely, if $L(S) \subseteq S$ and $L^{-1}(S) \subseteq S$, then it is clear that $L(S) = S$.
\end{proof}

\begin{theorem}\label{np}
Let $v,w \in \mathbb{R}^{n}$ such that $v$ contains both positive and negative entries and $w\neq 0$. Then there exists a nonnegative invertible matrix $B$ such that $Bv=w.$
\end{theorem}
\begin{proof}
Without loss of generality assume that, $v_1>0, v_n<0$ and $w_1\neq 0$. Now, we construct the matrix $B$ as follows:
%\textbf{Case 1:}

\textbf{Step 1:}
 First row of the matrix $B$ is constructed as follows:
\begin{itemize}
\item[(a)] If $w_1>0$, then the first row is $(\frac{w_1}{v_1},0, \ldots, 0)$.
\item[(b)] If $w_1<0$, then the first row is $(0, \ldots, 0, \frac{w_1}{v_n})$.
%\item[(c)] If $w_1=0$, then the first row is $(\frac{1}{v_1},0, \ldots, 0, \frac{-1}{v_n})$.
\end{itemize}
From the construction, it is clear that the first row of $B$ is nonnegative.

\textbf{Step 2:}
 Now, we construct $i^{th}$ row for $1<i<n$ as follows:
\begin{itemize}
\item[(a)] If $w_i>0$ and  $v_i>0$, then $i^{th}$ row is $(\frac{w_i}{2v_1},0,\ldots, \frac{w_i}{2v_i}, \ldots, 0).$
\item[(b)] If $w_i>0$ and  $ v_i<0$, then $i^{th}$ row is $(\frac{w_i +1}{v_1},0,\ldots, \frac{-1}{v_i}, \ldots, 0).$
\item[(c)] If $w_i>0$ and  $ v_i=0$, then $i^{th}$ row is $(\frac{w_i}{v_1},0,\ldots, 1, \ldots, 0).$
\item[(d)] If $w_i<0$ and  $ v_i>0$, then $i^{th}$ row is $(0,\ldots, \frac{1}{v_i}, \ldots, \frac{w_i-1}{v_n}).$
\item[(e)] If $w_i<0$ and  $ v_i<0$, then $i^{th}$ row is $(0,\ldots, \frac{w_i}{2v_i}, \ldots, \frac{w_i}{2v_n}).$
\item[(f)] If $w_i<0$ and  $ v_i=0$, then $i^{th}$ row is $(0,\ldots, 1, \ldots, \frac{w_i}{v_n}).$
\item[(g)] If $w_i=0$ and  $ v_i>0$, then $i^{th}$ row is $(0,\ldots, \frac{1}{v_i}, \ldots, \frac{-1}{v_n}).$
\item[(h)] If $w_i=0$ and  $ v_i<0$, then $i^{th}$ row is $(\frac{1}{v_1},0,\ldots, \frac{-1}{v_i}, \ldots, 0).$
\item[(i)] If $w_i=0$ and  $ v_i=0$, then $i^{th}$ row is $(0,\ldots, 1, \ldots, 0).$
\end{itemize}
For $2 \leq i \leq n-1$, in the above construction, the $i^{th}$ entry is nonzero only in the $i^{th}$ row. This establishes the linear independence of the rows from $1$ to $n-1$.

\textbf{Step 3:}
 Now, let us construct the $n^{th}$ row of the matrix $B$.
\begin{itemize}
\item[(a)] If $w_n>0$ and  $ w_1 < 0$, then the $n^{th}$ row is $(\frac{w_n}{v_1},0, \ldots, 0)$.
\item[(b)] If $w_n>0 $ and  $ w_1> 0$, then the $n^{th}$ row is $(\frac{w_n+1}{v_1},0, \ldots, \frac{-1}{v_n})$.
\item[(c)] If $w_n<0$ and  $ w_1< 0$, then the $n^{th}$ row is $(\frac{1}{v_1},0, \ldots, \frac{w_n-1}{v_n})$.
\item[(d)] If $w_n<0$ and  $ w_1> 0$, then the $n^{th}$ row is $(0, \ldots, \frac{w_n}{v_n})$.
\item[(e)] If $w_n=0$,  then the $n^{th}$ row is $(\frac{1}{v_1},0, \ldots, \frac{-1}{v_n})$.\\
\end{itemize}
From the construction, it is clear the rows of the matrix $B$ are linear independent, nonnegative and $Bv = w.$

%
%\textbf{Case 2:} Suppose $w_n=0$ and $w_1=0$.
%Construct the rows $1, \dots, k-1, k+1, \dots, n-1$ as in Case 1. Now, first let us construct the $k^{th}$ row as follows:
%\begin{itemize}
%\item[(1)] If $w_k>0$, then the $k^{th}$ row is $(\frac{w_k}{v_1},0, \ldots, 0)$.
%\item[(2)] If $w_k<0$, then the $k^{th}$ row is $(0, \ldots, 0, \frac{w_k}{v_n})$.
%\end{itemize}
%Finally,  we construct $n^{th}$ row as follows:
%\begin{itemize}
%\item[(1)] If $v_k<0$, then the $n^{th}$ row is $(\frac{1}{v_1},\ldots, \frac{-1}{v_k}, \ldots, 0)$.
%\item[(2)] If $v_k>0$, then the $n^{th}$ row is $(0, \ldots,\frac{1}{v_k}, \ldots, 0, \frac{-1}{v_n})$.
%\item[(3)] If $v_k=0$, then the $n^{th}$ row is $(0,\ldots, 1, \ldots, 0)$.
%\end{itemize}
%From the construction, it is clear the rows of the matrix $B$ are linear independent, nonnegative and $Bv = w.$
\end{proof}
Let us illustrate construction of the previous result, by an example.
\begin{ex}
Let $v=(1, 0, -5, -1)^T$ and $w=(3, 2, -10, 0)^T$. Then $v$ contains both positive and negative entries and $w\neq 0$. 
Since $v_1=1>0$ and $w_1=3>0$, the first row of $B$ is $(3,0,0,0)$.
Now $v_2=0$ and $w_2=2>0$, so the second row of $B$ is $(2,1,0,0)$. 
Again $v_3=-5<0$ and $w_3=-10<0$, so the third row of $B$ is $(0,0, 1,5)$. 
Finally $v_4=-1<0$ and $w_4=0$, so the last row of $B$ is $(1,0,0,1)$. Thus the matrix \begin{center} $B = \left(
\begin{array}{cccc}
3 & 0 & 0 & 0 \\
2 & 1 & 0 & 0\\
0 & 0 & 1 & 5\\
1  & 0 & 0 & 1\end{array} \right)$
\end{center}
is nonnegative, invertible and $Bv=w$.
\end{ex}

Let $e_i$ denote the vector in $\mathbb{R}^n$ whose $i^{th}$ entry is $1$ and others entries are zero.
\begin{theorem}\label{pos}
Let $v,w \in \mathbb{R}^{n}$ such that $v\geq 0$ and $w> 0$. Then there exists a nonnegative invertible matrix $B$ such that $Bv=w.$
\end{theorem}
\begin{proof}
Without loss of generality assume that  first $k$ entries of the vector $v$ are positive. Now, we construct the matrix $B$ as follows:
\begin{eqnarray}
Be_1&=&\left(\frac{w_1}{v_1},\frac{w_2}{2v_1},\ldots,\frac{w_k}{2v_1},\frac{w_{k+1}}{v_1},\frac{w_{k+2}}{v_1},\ldots, \frac{w_n}{v_1}\right)^T \nonumber \\
Be_2&=& \frac{w_2}{2v_2}e_2 \nonumber \\
\vdots \nonumber\\
Be_k&=& \frac{w_k}{2v_k}e_k \nonumber \\
Be_{k+1}&=& e_{k+1} \nonumber \\
\vdots \nonumber\\
Be_{n}&=& e_{n}. \nonumber
\end{eqnarray}
Thus $B$ is nonnegative lower triangular matrix whose diagonal entries are nonzero. So $B$ is invertible and $Bv=w.$
\end{proof}
Using Theorem \ref{np} and Theorem \ref{pos}, we give an alternate simple proof of Lemma 2.10 \cite{john-main} in the next result.
\begin{cor}
Let $v\in \mathbb{R}^{n}$ and $w\in \mathbb{R}^{m}$ with $n>m$. If $v$ contains both positive and negative entries, then there exists a nonnegative full row rank matrix $B$ such that $Bv=w.$ The same holds if $v\geq 0$ and $w>0$.
\end{cor}
\begin{proof}
Let $0\neq x\in \mathbb{R}^{n-m}$. Then, the vector  $w^\prime=(w^T,x^T)^T \in \mathbb{R}^{n}$ and
$w^\prime \neq 0$. By Theorem \ref{np}, there exists a nonnegative invertible matrix $B^\prime$ such that 
$B^\prime v=w^\prime$. Let $B^\prime=(B^T, C^T)^T$, where $B\in \mathbb{R}^{m \times n}$.
Then $B\geq 0$ and $Bv=w$.  Since $B^\prime$ is invertible,  $B$ has full row rank. 
Similarly one can proof second part using Theorem \ref{pos}.
\end{proof}
Next, we consider linear preserver of minimal semipositive matrices of the form $L(A)=XAY$,
for some fixed $X\in \mathbb{R}^{m\times m}$ and $Y\in \mathbb{R}^{n\times n}$.
In \cite[Theorem 2.11]{john-main}, the authors claimed that, if $m>n$, then $L$ is an into linear preserver of
minimal semipositive matrices if and only if either $X$ is monomial and $Y$ is inverse nonnegative or $-X$ is monomial
and $-Y$ is inverse nonnegative. But the converse in the above statement is not true. This is illustrated next.%  and $L$ is an onto linear preserver of  minimal semipositive matrices if and only if either $X$ and $Y$ are monomial or  $-X$ and $-Y$ are monomial \cite[Theorem 2.11]{john-main}
\begin{ex}
%Converse part of Theorem 2.11 \cite{john-main} is not true.
Let $X=\begin{pmatrix}
1 & 1\\ 1 & 1
\end{pmatrix}$ and $Y=[1]$. Since $2 \times 1$ minimal semipositive matrices are only positive column,
so $L(A)=XAY$ is into linear preserver of minimal semipositive matrices. But $X$ is not monomial.
\end{ex}

%\begin{remark}
%Let try to reformulate  Theorem 2.11 \cite{john-main} and prove.
%\end{remark}
In the next result, we establish a necessary and sufficient condition for $L(A)$ to be
an into linear preserver of  minimal semipositive matrices, when $m=n$.
\begin{lemma}\label{key1}

 Let $X \in \mathbb{R}^{n \times n}$ be such that both $X$ and $-X$ are not inverse nonnegative. Then there exists a vector
 $v \in \mathbb{R}^n$ such that $v \ngeq 0$, $v \nleq 0$ and $Xv \geq 0$.
\end{lemma}
\begin{proof}
Since $X^{-1}$ and $-X^{-1}$ are not  nonnegative, there exist vectors $u, w \in \mathbb{R}^n$ such that $Xu \geq 0$, $Xw \geq 0$, $u \ngeq 0$ and $w \nleq 0$. If either $u$  or $w$ contain both positive and negative entries, then we are done. Otherwise, $u \leq 0$ and $w \geq 0.$ Also, $u$ is not a linear multiple of $w$. Let $u=(u_1,\ldots,u_n)^T$ and $w=(w_1,\ldots,w_n)^T$. So, without loss of generality, let us assume  that the matrix  $D =\left[
                     \begin{array}{cc}
                       u_1 & w_1 \\
                       u_2 & w_2
                     \end{array}
                   \right]$ is invertible. If $det(D) > 0$, then define $\alpha = \frac{w_1 + w_2}{det(D)}$ and $\beta = - \frac{u_1 + u_2}{det(D)}$. Now, $D\left[
                     \begin{array}{c}
                       \alpha \\
                        \beta
                     \end{array}
                   \right] = \left[
                     \begin{array}{c}
                       1 \\
                       - 1
                     \end{array} \right]$. Set $v = \alpha u + \beta w$, then $v$ contains both positive and negative entries, and $Xv \geq 0$. If $det(D) < 0$, then the proof is similar.
\end{proof}

\begin{theorem}\label{mainresult1}
Let $X,Y \in \mathbb{R}^{n \times n}$. Let  $L: \mathbb{R}^{n \times n} \rightarrow \mathbb{R}^{n \times n}$
defined by $L(A)=XAY$.  Then $L$ is an into linear preserver of minimal semipositive matrices if and only if
$X,Y$ are inverse nonnegative matrices or $-X, -Y$ are inverse nonnegative matrices.
\end{theorem}
\begin{proof}
Suppose that $A$ is minimal semipositive. Thus $A^{-1}$ exists and $A^{-1}\geq 0$.
If $X$ and $Y$ are inverse nonnegative,  then $(XAY)^{-1}=Y^{-1}A^{-1}X^{-1}$ is nonnegative.
Thus $L(A)$ is minimal semipositive.

Conversely, suppose that $L(A)$ is an into linear preserver of minimal semipositive matrices. First, note that $X$ and $Y$
must be invertible. If either $X$ or $Y$ is not invertible, then $XAY$ is not invertible. Now, suppose that neither $X$ nor $-X$
is inverse nonnegative. Then, by Lemma \ref{key1}, there exists some vector $v \ngeq 0$ and $v\nleq 0$ such that $Xv\geq 0$.
Let $w\ngeq 0$, and consider $Yw$. By Theorem \ref{np}, there exists nonnegative invertible matrix $B$ such that $Bv=Yw$.
Let $A=B^{-1}$. Then $A$ is minimal semipositive. Now $L(A)w=XAYw=XABv=Xv\geq 0$. But $w\ngeq 0$, a contradiction.
Thus $X$ or $-X$ is inverse nonnegative.\\
Now, suppose that $X$ is inverse nonnegative and let $Y$ is not inverse nonnegative.
Then there exists some vector $w> 0$ such that $u=Y^{-1}w \ngeq 0$. Let $v=X^{-1}w$. Then $v \geq 0$ and $Xv>0$.
By Theorem \ref{pos}, there exists nonnegative invertible matrix $B$ such that $Bv=w$. Let $A=B^{-1}$.
Then $A$ is minimal semipositive. Now $L(A)u=XAY(Y^{-1}w)=XAw=Xv>0$. Since $L(A)^{-1} \geq 0$, we have $u \geq 0$,
a contradiction.  Thus $Y$ must be inverse nonnegative. Similarly if $-X$ is inverse nonnegative then $-Y$ is inverse nonnegative.
\end{proof}
Now, we present a necessary and sufficient condition for $L(A)$ to be an onto linear preserver of  minimal semipositive matrices, when $m=n$. First, we mention a result that will be used in the proof.
\begin{theorem}\label{msp-basis}\cite[Theorem 3.1]{prokcs5}
For $m \geq n$, the set of $m \times n$ minimal semipositive matrices contains a basis for the set of all $m \times n$ matrices.
\end{theorem}

\begin{theorem}\label{mainresult2}
Let $X,Y \in \mathbb{R}^{n \times n}$. Let $L(A)=XAY$ be a linear operator on $\mathbb{R}^{n \times n}$. Then $L(A)$ is an onto linear preserver of minimal semipositive matrices if and only if $X,Y$ are monomial matrices or $-X, -Y$ are monomial matrices.
\end{theorem}
\begin{proof}
Let $L$ be an onto linear preserver of minimal semipositive matrices.
Then $L$ is into linear preserver of minimal semipositive matrices. Since, by Theorem \ref{msp-basis},
the set of all minimal semipositive matrices contains a basis for the set of all $n \times n$ matrices,
by Lemma \ref{basis-onto-into}, we have onto preservers are same as the invertible into preservers of minimal semipositive
matrices whose inverse is also an  into preserver of minimal semipositive matrices. Thus, by Theorem \ref{mainresult1}, the matrices $X, Y$ are nonnegative
and inverse nonnegative or the matrices $-X, -Y$ are nonnegative and inverse nonnegative. %By Theorem \ref{mono-char},
Hence X, Y are monomial matrices or - X,-Y are monomial matrices. Converse is easy to verify.
\end{proof}
Next, we would like to point out an error in the proof of Theorem 2.4 \cite{john-main}.
Let $ X\in \mathbb{R}^{n \times n}$. Suppose that neither $X$ nor $-X$ is monomial.
The authors in \cite{john-main} used the fact that there exists a vector $v>0$ such that $Xv$ has some zero entry.
The following example shows that this is incorrect.

\begin{ex}
Let $X = \left(
\begin{array}{ccc}
1 & 0 & 0 \\
0 & -1 & 0\\
1  & 1 & 1\end{array} \right).$ Then neither $X$ nor $-X$ is monomial and there does not exists any $v>0$ such that $Xv$ contains a zero entry.
\end{ex}
Let $e$ denote the vector in $\mathbb{R}^n$ whose all entries are 1. In the next result, we give a correct proof of Theorem 2.4 \cite{john-main}.

\begin{theorem}
Let $X\in \mathbb{R}^{m \times m}$ and $Y \in \mathbb{R}^{n \times n}$. Let $L(A)=XAY$ be a linear operator
on $\mathbb{R}^{m \times n}$. Then $L(A)$ is an into linear preserver of semipositive matrices if and only if either
$X$ is row positive and $Y$ is inverse nonnegative or $-X$ is row positive and $-Y$ is inverse nonnegative.
\end{theorem}
\begin{proof}
Suppose that $A$ is semipositive. Then there exists a semipositivity vector $v$ such that $Av>0$. If $X$ is row positive and $Y$ is inverse nonnegative, then $u=Y^{-1}v>0$ and $XAYu>0$. Thus $L(A)$ is semipositive. If $-X$ is row positive and $-Y$ is inverse nonnegative, then $L(A)=(-X)A(-Y)=XAY$ is semipositive.

Conversely, suppose that $L(A)$ is an into linear preserver of semipositive matrices. Suppose that neither $X$ nor $-X$ is row positive.

Case (i): If $X$ contains a zero row, then $XAY$ contains a zero row for any $A$ and $Y$. Thus $XAY$ is not semipositive.

Case (ii): If $X$ contains a non-zero row $i$ with both positive and negative entries, then we can construct a positive vector $v$ such that $(Xv)_i=0$. Set $A=(v,\ldots,v)$. Then $A$ is semipositive and $XAY$ contains a zero row and hence $XAY$ is not semipositive.

Case (iii): If $X$ does not contains zero row or non-zero row such that  it contains both positive and negative entries.
Then there exists a row which is non positive. Since, $-X$ is also not row positive, so $X$ contains a nonnegative row.
Without loss of generality, assume that first row of $X$ is non positive and second row of $X$ is nonnegative.
Let $A=(e,0,\ldots, 0)$. Then $A$ is semipositive and $XA=(Xe,0,\ldots,0)$. Let $Xe=(\alpha_1,\alpha_2,\ldots, \alpha_n)^T$. Then $\alpha_1<0$ and $\alpha_2>0$. Let $u\in \mathbb{R}^n$ and $Yu=v$. Then $XAYu=(\alpha_1v_1,\alpha_2v_1,\ldots, \alpha_nv_1)$. Thus $XAYu$ never be positive. Thus $XAY$ is not semipositive.
Hence either $X$ or $-X$ is row positive.

Now, suppose that $X$ is row positive. We first prove that $Y$ is invertible. Let $v=(v_1,\ldots , v_n)^T$ be
such that $v^TY=0$. Suppose that $v\neq 0$. With out loss of generality, we can  assume that $v_1>0$. Let $A$ be the matrix
whose rows are the row vector $v^T$. Since the first column of $A$ is positive, the matrix $A$ is semipositive. However,
by definition the matrix $AY$ is zero so that $XAY$ is not semipositive. This is a contradiction. Hence $Y$ is invertible.
Next, let us show that $C=Y^{-1}\geq 0$. On  contrary, suppose that $c_{ij}<0$ for some $i,j$. Let $A$ be the matrix whose
rows are the negative of the $i^{th}$ row of $C$. Since the $j^{th}$ column of $A$ is positive, the matrix $A$ is
semipositive. However the matrix $AY$ is nonpositive and hence $XAY$ is not semipositive, contradiction to the assumption.
Thus $Y^{-1}\geq 0$. A similar argument shows that if $-X$ is row positive, then $-Y$ is inverse nonnegative.
\end{proof}

%Problems to consider, other then we mentioned in side the paper.
%\begin{enumerate}
%\item It seems there is a mistake in the statement of Theorem $3.5$
%\item Invertible into preservers on SP matrices. Are they of the form $XAY$?
%\item general into and onto linear preservers on minimal semipositive matrices.
%\end{enumerate}

\bibliographystyle{amsplain}
\bibliography{semi-preserver}
\end{document}